\documentclass[11pt,b5paper,twoside,headrule]{amsart}
\usepackage{document}
\usepackage[all]{xy}
\usepackage{amssymb}
\newtheorem{dummy}{anything}[section]
\newtheorem{theorem}[dummy]{Theorem}
\newtheorem*{thma}{Theorem A}

\newtheorem*{corb}{Corollary B}
 
\newtheorem{lemma}[dummy]{Lemma}

\newtheorem{corollary}[dummy]{Corollary}

\theoremstyle{definition}

 \newtheorem{example}[dummy]{Example}

\renewcommand{\labelenumi}{(\roman{enumi})}

\newcommand
{\eqncount}{\setcounter{equation}{\value{dummy}}%
\addtocounter{dummy}{1}}

\newcommand{\cG}{\mathcal G}
\newcommand{\cH}{\mathcal H}

\newcommand{\cM}{\mathcal M}

\newcommand{\bZ}{\mathbf Z}

\newcommand{\bbZ}{\mathbf Z}
\newcommand{\bbN}{\mathbb N}

\newcommand{\fP}{\mathfrak P}
\newcommand{\fE}{\mathfrak E}

\newcommand{\cy}[1]{\bZ/{#1}}

\newcommand{\vv}{\, | \,}

\newcommand{\mmatrix}[4]{\left (\vcenter
{\xymatrix@C-2pc@R-2pc{#1&#2\\#3&#4} }
\right )}

\DeclareMathOperator{\Ind}{Ind}
\DeclareMathOperator{\ind}{Ind}
\DeclareMathOperator{\Res}{Res}
\DeclareMathOperator{\res}{Res}

\DeclareMathAlphabet\EuR{U}{eur}{m}{n}
\SetMathAlphabet\EuR{bold}{U}{eur}{b}{n}

 \DeclareMathOperator{\aut}{Aut}

   \DeclareMathOperator{\Nil}{Nil}
   \DeclareMathOperator{\NK}{N\hskip-.5pt K}
    \DeclareMathOperator{\NIL}{NIL}
     \DeclareMathOperator{\FP}{\mathcal{P}}
 \newcommand{\id}{\operatorname{id}}
\newcommand{\la}{\langle}
\newcommand{\ra}{\rangle}
\DeclareMathOperator{\pr}{pr}
\newcommand{\hsp}[1]{{\hphantom{#1}}}
\newcommand{\clunk}[2]{{\hbox to #1 truecm{\hfill$#2$\hfill}}}
 
\begin{document}

\def\evenhead{{\protect\centerline{\textsl{\large{Ian Hambleton and Wolfgang L\"uck}}}\hfill}}

\def\oddhead{{\protect\centerline{\textsl{\large{Bass Nil Groups}}}\hfill}}

\pagestyle{myheadings} \markboth{\evenhead}{\oddhead}

\thispagestyle{empty} \noindent{{\small\rm Pure and Applied
Mathematics Quarterly\\ Volume 8, Number 1\\ (\textit{Special Issue:
In honor of \\ F. Thomas Farrell and Lowell E. Jones, Part 1 of 2})\\
199---219, 2012} \vspace*{1.5cm} \normalsize

\begin{center}
\Large{\bf Induction and Computation of Bass Nil Groups for Finite
Groups}
\end{center}

\begin{center}
{\large Ian Hambleton and Wolfgang L\"uck}
\end{center}
\footnotetext{Received July 5, 2007.}

\footnotetext{Research partially supported by NSERC Discovery Grant
A4000 and by the Max Planck Forschungspreis of the second author.
The first author would like to thank the SFB 478, Universit\"at
M\"unster, for its hospitality and support.}

\renewcommand{\thefootnote}{\fnsymbol{footnote}}

\bigskip

\begin{center}
\begin{minipage}{5in}
\noindent{\bf Abstract:} Let $G$ be a finite group. We show that the
Bass $\Nil$-groups $\NK_n(RG)$, $n \in \bZ$, are generated from the
$p$-subgroups of $G$ by induction, certain twistings maps depending
on elements in the centralizers of the $p$-subgroups, and the
Verschiebung homomorphisms.  As a consequence, the groups
$\NK_n(RG)$ are generated by induction from elementary subgroups.
For $\NK_0(\bZ G)$ we get an improved estimate of the torsion
exponent.
\end{minipage}
\end{center}

\section{Introduction}\label{one}
In this note we study the  Bass $\Nil$-groups (see \cite[Chap.~XII]{bass1})
$$\NK_n(R G) = \ker(K_n(R G[t]) \to K_n(R G)),$$
where $R$ is an associative ring with unit, $G$ is a finite group, $n \in \bZ$, and the augmentation map sends $t\mapsto 0$. Alternately, the isomorphism
$$\NK_n(RG) \cong \widetilde K_{n-1}(\NIL(RG))$$
identifies the Bass $\Nil$-groups with the  $K$-theory of the
category $\NIL(RG)$ of nilpotent endomorphisms $(Q, f)$ on finitely-generated projective $RG$-modules \cite[Chap.~XII]{bass1}, \cite[Theorem 2]{grayson2}. Farrell \cite{farrell1} proved that $\NK_1(RG)$ is not finitely-generated as an abelian group whenever it is non-zero, and the corresponding result holds for $\NK_n(RG)$,  $n\in \bZ$ (see \cite[4.1]{weibel1}), so some organizing principle is needed to better understand the structure of the $\Nil$-groups. Our approach is via induction theory.
The functors $\NK_n$ are Mackey functors on the subgroups of $G$, and we ask to what extent they can be computed from the $\Nil$-groups of proper subgroups of $G$.

The Bass-Heller-Swan formula \cite[Chap.~XII, \S 7]{bass1}, \cite[p.~236]{grayson1} relates the Bass $\Nil$-groups with the $K$-theory of the infinite group $G \times \bZ$.
There are two (split) surjective maps
$$N_{\pm} \colon K_n(R[G \times \bZ]) \to \NK_n(R G)$$
which form part of the Bass-Heller-Swan direct sum decomposition
$$ K_n(R[G \times \bZ])=K_n(R G)\oplus K_{n-1}(R G)\oplus \NK_n(R G)\oplus \NK_n(R G).$$
Notice that both $K_n(R[(-)\times \bZ])$ and $\NK_n(R[-])$ are  Mackey functors on the subgroups of $G$ (see Farrell-Hsiang \cite[\S 2]{farrell-hsiang2} for this observation about infinite groups).
We observe that the maps $N_{\pm}$ are actually natural transformations of Mackey functors (see Section \ref{seven}). It follows from Dress induction \cite{dress2} that the functors $\NK_n(RG)$ and the maps $N_{\pm}$ are computable from the hyperelementary family (see Section \ref{four}, and Harmon \cite[Cor.~4]{harmon1} for the case $n=1$).
We will show how the results of Farrell \cite{farrell1} and techniques of Farrell-Hsiang \cite{farrell-hsiang3} lead to a better generation statement  for the Bass $\Nil$-groups.

\smallskip
We need some notation to state the main result.
For each prime $p$, we denote by $\fP_p(G)$ the set of finite $p$-subgroups of $G$, and by $\fE_p(G)$ the set of $p$-elementary subgroups of $G$. Recall that a $p$-elementary group has the form $E=C \times P$, where $P$ is a finite $p$-group, and $C$ is a finite cyclic group of order prime to $p$. For each element $g\in C$, we let
$$I(g) = \{ k \in \bbN \vv \text{\ a\ prime\ } q \text{\ divides\ } k \Rightarrow q \text{\ divides\ } |g|\}$$
where $|g|$ denotes the order of $g$. For each $P \in \fP_p(G)$, let
$$C_G^\perp(P) = \{ g\in G \vv gx=xg,  \forall x \in P,\text{\ and\ } p\nmid |g|\}$$
and for each $g\in C_G^\perp(P)$ we define a functor
$$\phi(P,g)\colon \NIL(RP) \to \NIL(RG)$$
by sending a nilpotent $RP$-endomorphism $f\colon Q\to Q$ of a finitely-generated projective $RP$-module $Q$ to the nilpotent $RG$-endomorphism
$$RG\otimes_{RP} Q \to RG\otimes_{RP} Q, \qquad x\otimes q \mapsto xg\otimes f(q)\ .$$
Note that this $RG$-endomorphism is well-defined since $g \in C_G^\perp(P)$.
The functor $\phi(P,g)$ induces a homomorphism $$\phi(P,g)\colon \NK_n(RP) \to \NK_n(RG)$$ for each $n\in \bZ$.
For each $p$-subgroup $P$ in $G$, define a homomorphism
$$\Phi_P\colon \NK_n(RP) \to \NK_n(RG)$$
by the formula
$$\Phi_P = \sum_{{g\in C_G^\perp(P), \ k \in I(g)}} {\hskip -6pt}V_k \circ \phi(P,g),$$
where
$$V_k \colon \NK_n(RG) \to \NK_n(RG)$$
denotes the Verschiebung homomorphism, $k \geq 1$, recalled in more detail in Section \ref{two}.

\begin{thma}
Let $R$ be an associative ring with unit, and $G$ be a finite group. For each prime $p$, the map
$$\Phi = (\Phi_P)\colon  \bigoplus_{P \in \fP_p(G)} \NK_n(RP)_{(p)} \to \NK_n(RG)_{(p)}$$
is surjective for all $n\in \bZ$, after localizing at $p$.
\end{thma}
For every $g\in  C_G^\perp(P)$, the homomorphism $\phi(P,g)$ factorizes as
$$\NK_n(RP) \xrightarrow{\phi(P,g)} \NK_n(R[C \times P]) \xrightarrow{i_*} \NK_n(RG)$$
where $C = \la g \ra$ and $i\colon C\times P \to G$ is the inclusion map.  Since the Verschiebung homomorphisms are natural with respect to the maps induced by group homomorphisms, we obtain:
\begin{corb}
The sum of the induction maps
$$\bigoplus_{E\in \fE_p(G)} \NK_n(RE)_{(p)} \to \NK_n(RG)_{(p)}$$
from $p$-elementary subgroups is surjective, for all $n\in \bZ$.
\end{corb}
Note that Theorem A does not show that $\NK_n(RG)$ is generated by induction from $p$-groups, because the maps $\phi(P,g)$ for $g\neq 1$ are not the usual maps induced by inclusion $P \subset G$.

\smallskip
The Bass $\Nil$-groups are non-finitely generated torsion groups (if non-zero) so they remain difficult to calculate explicitly, but we have some new qualitative results. For example, Theorem A shows that the order of every element of $\NK_n(R G)$ is some power of $m=|G|$, whenever $\NK_n(R) =0$ (since its $q$-localization is zero for all $q\nmid m$).
For $R = \bZ$ and some related rings, this is a result of Weibel~\cite[(6.5), p.~490]{weibel1}.
In particular, we know that every element of $\NK _n(\bZ P)$ has $p$-primary order, for every  finite $p$-group $P$.
 If $R$ is a regular (noetherian) ring (e.g.~$R=\bZ$), then $\NK_n(R) = 0$ for all $n\in \bZ$.

Note that an exponent that holds uniformly for all elements in $\NK_n(RP)$, over all $p$-groups of $G$,  will be an exponent for $\NK_n(RG)$. As a special case, we have
 $\NK _n(\bZ[\bZ/p])=0$ for $n \le 1$, for  $p$ a prime (see
Bass-Murthy \cite{bass-murthy1}), so Theorem A
 implies:

\begin{corollary} \label{cor:vanishing_of_NK_n(ZG)_for_n_le_1}
  Let $G$ be a finite group and let $p$ be a prime. Suppose that $p^2$ does not divide the
  order of $G$.  Then
  $$\NK _n(\bZ G)_{(p)} = 0$$
for $n \le 1$.
\end{corollary}

As an application, we get a new proof of the fact that $\NK _n(\bZ G) = 0$, for $n \le 1$, if the order of $G$ is square-free (see Harmon \cite{harmon1}).

We also get from Theorem A an improved estimate on the exponent of $\NK _0(\bZ G)$, using a result of Connolly-da Silva \cite{connolly-dasilva1}. If $n$ is a positive integer, and $n_q=q^k$ is its $q$-primary part, then let  $c_q(n) = q^l$, where $l \geq \log_q(kn)$.
According to \cite{connolly-dasilva1}, the exponent of $\NK_0(\bZ G)$ divides
$$c(n) = \prod_{q\mid n} c_q(n), \quad \text{where\ } n = |G|,$$
but according to Theorem A, the exponent of $\NK_0(\bZ G)$ divides
$$d(n) = \prod_{q\mid n} c(n_q)\ .$$
For example, $c(60) = 1296000$, but $d(60) = 120$.

\medskip
\noindent
{\textbf{Acknowledgement}}: This paper is dedicated to Tom Farrell and Lowell Jones, whose work in geometric topology has  been a constant inspiration. We are also indebted to Frank Quinn, who suggested that the Farrell-Hsiang induction techniques should be re-examined (see \cite{quinn1}).
\section{Bass $\Nil$-groups}\label{two}
Standard constructions in algebraic $K$-theory for exact categories or Waldhausen
categories yield only $K$-groups in degrees $n \ge 0$ (see Quillen \cite{quillen1},
Waldhausen~\cite{waldhausen1}).  One approach on the categorial level to negative
$K$-groups has been developed by Pedersen-Weibel (see \cite{pedersen-weibel2}, \cite[\S 2.1]{hp2004}).  Another
ring theoretic approach is given as follows (see
Bartels-L\"uck~\cite[Section~9]{bartels-lueck1}, Wagoner~\cite{wagoner1}).  The
{\em cone ring} $\Lambda \bbZ$ of $\bbZ$ is the ring of column and row finite $\bbN \times
\bbN$-matrices over $\bbZ$, i.e., matrices such that every column and every row contains
only finitely many non-zero entries. The {\em suspension ring} $\Sigma \bbZ$ is the
quotient of $\Lambda \bbZ$ by the ideal of finite matrices. For an associative (but not necessarily commutative) ring $A$ with
unit,  we define $\Lambda A = \Lambda \bbZ \otimes_{\bbZ} A$ and $\Sigma A = \Sigma
\bbZ \otimes_{\bbZ} A$.  Obviously $\Lambda$ and $\Sigma$ are functors from the category
of rings to itself.  There are identifications, natural in $A$,
\eqncount
\begin{eqnarray}
K_{n-1}(A) & = & K_n(\Sigma A) \label{K_n-1(A)_is_K_n(SigmaA)}
\\
\eqncount
\NK_{n-1}(A)& = & \NK_n(\Sigma A) \label{NK_n-1(A)_is_NK_n(SigmaA)}
\end{eqnarray}
 for all $n \in \bbZ$. In our applications, usually $A = RG$ where $R$ is a ring with unit and $G$ is a finite group.

Using these identifications it is clear how to extend the definitions of certain maps between
$K$-groups given by exact functors to negative degrees.  Moreover, we will
 explain constructions and proofs of the commutativity of certain diagrams only for
$n \ge 1$, and will not explicitly mention that these carry over to all $n \in \bbZ$,
because of the identifications~\eqref{K_n-1(A)_is_K_n(SigmaA)} and
\eqref{NK_n-1(A)_is_NK_n(SigmaA)} and the obvious identification $\Sigma(RG) = (\Sigma
R)G$, or because of Pedersen-Weibel~\cite{pedersen-weibel2}.

We have a direct sum decomposition
\eqncount
\begin{eqnarray}K_n(A[t])  & = & K_n(A) \oplus \NK_n(A)
\label{K_n(A[t])_is_NK_n(A)_oplus_K_N(A)}
\end{eqnarray}
which is natural in $A$, using the inclusion $A \to A[t]$ and the ring map  $A[t] \to A$ defined by $t\mapsto 0$.

Let $\FP(A)$ be the exact category of finitely generated projective $A$-modules, and let
$\NIL(A)$ be the exact category of nilpotent endomorphism of finitely generated projective
$A$-modules.  The functor $\FP(A) \to \NIL(A)$ sending $Q\mapsto (Q,0)$  and the
functor $\NIL(A) \to \FP(A)$ sending $(Q,f)\mapsto Q$ are exact functors. They
yield a split injection on the $K$-groups $K_n(A):= K_n(\FP(A)) \to K_n(\NIL(A))$ for $n
\in \bbZ$.  Denote by $\widetilde{K}_n(\NIL(A))$ the cokernel for $n \in \bbZ$.  There is
an identification (see Grayson~\cite[Theorem~2]{grayson2})
\eqncount
\begin{eqnarray}
 \widetilde{K}_{n-1}(\NIL(A)) & = & \NK_{n}(A),
\label{K_n(Nil(A))_is_NK_nplus1(A)}
\end{eqnarray}
 for $n \in \bbZ$,
essentially given by the passage from a nilpotent $A$-endomorphism $(Q,f)$  to the $A\bbZ$-automorphism
$$A\bbZ \otimes_A Q \to A\bbZ \otimes_A Q, \quad u \otimes q \mapsto u \otimes q - ut
\otimes f(q),$$
for $t \in \bbZ$ a fixed generator.

The Bass $\Nil$-groups appear in the Bass-Heller-Swan decomposition for $n \in \bbZ$ (see  \cite[Chapter~XII]{bass1}, \cite{bass-heller-swan1}, \cite[p.~236]{grayson1}, \cite[p.~38]{quillen1}, and \cite[Theorem~10.1]{swan3}
for the original sources, or the expositions in \cite[Theorems~3.3.3 and 5.3.30]{rosenberg1}, \cite[Theorem~9.8]{srinivas1}
).
\eqncount
\begin{eqnarray}
B \colon K_n(A) \oplus K_{n-1}(A) \oplus \NK_n(A) \oplus \NK_n(A)
\xrightarrow{\cong} K_n(A\bbZ).
\label{Bass-Heller-Swan_decomposition}
\end{eqnarray}
The isomorphism $B$ is natural  in $A$ and comes from the localization sequence
\eqncount
\begin{multline}
  0 \to K_n(A) \xrightarrow{K_n(i) \oplus -K_n(i)} K_n(A[t]) \oplus K_n(A[t])
  \xrightarrow{K_n(j_+) + K_n(j_-)} K_n(A\bbZ)\\
  \xrightarrow{\partial_n} K_{n-1}(A) \to 0
\label{localization_sequence}
\end{multline}
where $i \colon A \to A[t]$ is the obvious inclusion and the inclusion $j_{\pm} \colon
A[t] \to A\bbZ$ sends $t$ to $t^{\pm 1}$ if we write $A\bbZ = A[t,t^{-1}]$, the splitting
of $\partial_n$
\eqncount
\begin{eqnarray}
s_n \colon K_{n-1}(A) \to K_n(A\bbZ)
\label{splitting_s_n}
\end{eqnarray}
which is given by the natural pairing
\begin{eqnarray*}
K_{n-1}(A) \otimes K_1(\bbZ[t,t^{-1}]) \to K_n(A \otimes_{\bbZ}\bbZ[t,t^{-1}]) = K_n(A\bbZ)
\end{eqnarray*}
evaluated at the class of unit $t \in \bbZ[t,t^{-1}]$ in $K_1(\bbZ[t,t^{-1}])$, and the
canonical splitting~\eqref{K_n(A[t])_is_NK_n(A)_oplus_K_N(A)}.
Let $B$ be the direct sum of  $K_n(j_+\circ i)$,  $s_n$, and the restrictions of the maps $ K_n(j_+)$ and $K_n(j_{-})$ to $\NK_n(A)$.

In particular we get two homomorphisms, both natural in $A$, from the Bass-Heller-Swan
decomposition~\eqref{Bass-Heller-Swan_decomposition}\begin{eqnarray*}
i_n \colon \NK_n(A) & \to & K_n(A\bbZ)\\
r_n \colon K_n(A\bbZ) & \to & \NK_n(A),
\end{eqnarray*}
 by focusing on the first copy of
$\NK_n(A)$,
such that $r_n \circ i_n$ is the identity on $\NK_n(A)$.

Let $\sigma_k \colon \bbZ \to
\bbZ$ be the injection given by $t\mapsto t^k$. We may consider the ring $A[t]$ as an $A[t]-A[t]$ bimodule  with standard left action, and right action $a(t)\cdot b(t) = a(t)b(t^k)$ induced by $\sigma_k$.
This map induces an induction functor
$$\ind_k\colon \FP(A[t]) \to \FP(A[t])$$ defined by
$P \mapsto A[t]\otimes_{\sigma_k} P$. There is also a restriction functor $$\res_k\colon \FP(A[t]) \to \FP(A[t])$$ defined by
equipping $P$ with the new $A[t]$-module structure
$a(t)\cdot p = a(t^k)p$, for all $a(t) \in A[t]$ and all $p \in P$.
 The
 induction and
restriction functors yield two homomorphisms
\begin{eqnarray*}
 \ind_{k} \colon K_n(A\bbZ) & \to & K_n(A\bbZ)
\\
\res_{k } \colon K_n(A\bbZ) & \to & K_n(A\bbZ)\ .
\end{eqnarray*}
See \cite{stienstra1} or \cite[p.~27]{quillen1} for more details.

There are also \emph{Verschiebung} and  \emph{Frobenius} homomorphisms
\eqncount
\begin{eqnarray}
V_k, F_k \colon \NK_n(A) & \to & \NK_n(A)
\label{F_k_and_V_k}
\end{eqnarray}
 induced on the $\Nil$-groups (and related to $\ind_k$ and $\res_k$ respectively).
The Frobenius  homomorphism is induced by the functor
$\NIL(A) \to \NIL(A)$ sending
$$(f \colon Q \to Q) \mapsto (f^k \colon Q \to Q),$$
 while the Verschiebung homomorphism is induced by the functor
$$\bigoplus_{i=1}^k Q \to \bigoplus_{i=1}^k Q,\quad (q_1, q_2, \ldots q_k) \mapsto
(f(q_k),q_1, q_2,\ldots, q_{k-1})\ .$$

The next result is proven by
Stienstra~\cite[Theorem~4.7]{stienstra1}). (Notice that
Stienstra considers only commutative rings $A$, but his argument
goes through in our case since the set of polynomials $T$ we
consider is $\{t^n \vv n \in \bbZ, n \ge 0\}$ and each
polynomial in $T$ is central in $A[t]$ with respect to the
multiplicative structure.)
\begin{lemma}\label{lem:ind/res_and_V/F}
  The following diagrams commute for all $n \in \bbZ$ and $k\in \bbZ, k \ge 1$
  $$\xymatrix{\NK_n(A)\ar[r]^{i_n}\ar[d]^{F_k}&{K_n(A\bbZ)}\ar[d]^{\res_{k}}\\
      {\NK_n(A)}\ar[r]^{i_n}&{K_n(A\bbZ)}}$$
       and
 $$\xymatrix{\NK_n(A)\ar[r]^{i_n}\ar[d]^{V_k}&{K_n(A\bbZ)}\ar[d]^{\ind_k}\\
    {\NK_n(A)}\ar[r]^{i_n}&{K_n(A\bbZ)}}$$
\end{lemma}

The next result is well-known for $n=1$ (see Farrell \cite[Lemma~3]{farrell1}). The general case is discussed by Weibel \cite[p.~479]{weibel1}, Stienstra \cite[p.~90]{stienstra1}, and  Grunewald \cite[Prop.~4.6]{grunewald1}.
\begin{lemma} \label{lem:Frobenius_finally_vanishes}
  For every $n \in \bbZ$ and each $x \in \NK_n(A)$, there exists a positive integer $M(x)$ such
  that $ \res_{m} \circ\, i_n(x) = 0$ for $m \ge M(x)$.
\end{lemma}
\begin{proof}

  The Frobenius homomorphism $F_m \colon \NK_n(A) \to \NK_n(A)$ is induced by the functor
  sending $(Q,f)$ in $\NIL(A)$ to $(Q,f^m)$. For a given object $(Q,f)$ in $\NIL(A)$ there
  exists a positive integer $M(f)$ with $(Q,f^m) = (Q,0)$ for $m \ge M(f)$. This implies
  by a filtration argument
  (see \cite[p.~90]{stienstra1} or
  \cite[Prop.~4.6]{grunewald1}), that
  for $x \in\NK_n(A)$ there exists a positive integer $M(x)$ with $F_m(x) = 0$ for $m \ge
  M(x)$.
  Now the claim follows from Lemma~\ref{lem:ind/res_and_V/F}.
\end{proof}

\section{Subgroups of $G \times \bZ$}\label{three}
A finite group $G$ is called $p$-hyperelementary if is isomorphic to an extension
$$1\to C \to G \to P \to 1$$
where $P$ is a $p$-group, and $C$ is a cyclic group of order prime to $p$. Such an extension is a semi-direct product, and hence determined by the action map $\alpha\colon P \to \aut(C)$ defined by conjugation. The group $G$ is $p$-elementary precisely when $\alpha$ is the trivial map, or in other words, when there exists a retraction $G \to C$.
Notice that for a cyclic group $C=\cy{q^k}$, where $q\neq p$ is a prime, $\aut(C) = \cy{q^{k-1}(q-1)}$, if $q$ odd, or
$\aut(C) = \cy{2^{k-2}} \times \cy{2}$, $k\geq 2$,  if $q=2$.
In either case, $\aut(C)_{(p)} \cong \aut(Q)_{(p)}$ by projection to any non-trivial
quotient group $C \to Q$.

\begin{lemma}
\label{lem_hyper_elem_versus_elem}
Let $p$ be a prime, and let $G$ be a finite  $p$-hyper\-elemen\-ta\-ry group. Suppose
that for every prime $q\neq p$ which divides the order of $G$, there
exists an epimorphism $f_q \colon G \to Q_q$ onto a non-trivial cyclic group $Q_q$ of
$q$-power order.  Then $G$ is $p$-elementary.
\end{lemma}
\begin{proof}
Let $Q$ be the product of the groups $Q_q$ over all primes $q\neq p$ which divide the order of $G$. Let $f\colon G \to Q$ be the product of the given epimorphisms.
Since every subgroup in $G$ of order prime to $p$ is characteristic, we have a diagram
$$\xymatrix@R-3pt{1 \ar[r] & C \ar[r]\ar[d] & G \ar[r]\ar[d] & P \ar@{=}[d]\ar[r]& 1\ \hphantom{.}\cr
1 \ar[r] & Q \ar[r] & \bar{G} \ar[r]\ & P \ar[r]& 1\ .
}$$
But the epimorphism $f\colon G \to Q$ induces a retraction $\bar{G} \to Q$ of the lower sequence, hence its action map $\bar{\alpha}\colon P \to \aut(Q)$ is trivial. As remarked above, this implies that $\alpha$ is also trivial and hence $G$ is $p$-elementary.
\end{proof}
We now combine this result with the techniques of \cite{farrell-hsiang3}. Given positive integers $m$, $n$  and a prime $p$, we choose
an integer $N=N(m,n,p)$ satisfying the following conditions:
\begin{enumerate}\addtolength{\itemsep}{0.2\baselineskip}
\item  $p\nmid N$, but $q\mid N$ if and only if $q\mid n$ for all primes $q\neq p$.
\item $k \geq \log_q(mn)$ (i.e.~$q^k \geq mn$) for each full prime power $q^k\|N$.
\end{enumerate}
The Farrell-Hsiang technique is to compute $K$-theory via $p$-hyperelementary subgroups $H \subset G \times \cy N$, and their inverse images $\Gamma_H = \pr^{-1}(H) \subset G \times \bZ$ via the second factor projection map $\pr\colon G\times \bZ \to G \times \cy N$.
\begin{lemma} \label{lem:deep_and_p-torsion}
  Let $G$ be a finite group and let $M$ be a positive integer.  Let $p$ be a prime
  dividing the order of $|G|$, and choose an integer $N= N(M, |G|,p)$.   For every $p$-hyperelementary subgroup $H \subset G \times \cy N$, one of the following holds:
  \begin{enumerate}
\item  the inverse image $\Gamma_H \subset G \times m\cdot\bZ$, for some $m \geq M$, or
 \item \label{lem:deep_and_p-torsion:p-elementary} the group $H$ is $p$-elementary.
\end{enumerate}
In the second case, we have the following additional properties:

 \begin{enumerate}\renewcommand{\theenumi}{\alph{enumi}}
 \renewcommand{\labelenumi}{(\theenumi)}
 \item \label{lem:deep_and_p-torsion:identification_of_H} There exists a finite $p$-group
   $P$ and isomorphism
   $$\alpha \colon P \times \bbZ \xrightarrow{\ \cong\ } \Gamma_H.$$

 \item \label{lem:deep_and_p-torsion:commutative_square} There exists a positive integer
   $k$, a positive integer $\ell$ with $(\ell,p) =1$, an element $u \in \cy \ell$ and an injective
   group homomorphism
   $$j \colon P \times \cy \ell \to G$$
   such that the following diagram commutes
   $$\xymatrix{\hsp{xxxx}{P \times \bbZ}\hsp{xxxx} \ar[r]^(0.6){\alpha}\ar[d]_{\id_P \times \beta}&\ {\Gamma_H}\ar[d]^{i} \\
   {P \times
     \cy \ell \times \bbZ}\ar[r]^(0.6){j \times k \cdot \id_{\bbZ}}&\ {G \times \bbZ}}$$
  where $i \colon
   \Gamma_{H} \to G \times \bbZ$ is the inclusion and $\beta \colon \bbZ \to \cy \ell
   \times \bbZ$ sends $n$ to $(nu,n)$.
\end{enumerate}
\end{lemma}
\begin{proof} In the proof we will write elements in $\bZ/N$ additively and elements in $G$ multiplicatively.  Let $H \subset G \times \cy N$ be a $p$-hyperelementary subgroup, and suppose that $\Gamma_H$ is \emph{not} contained in $G \times m\cdot \bZ$ for any $m\geq M$.
We have a pull-back diagram
$$\xymatrix@R-2pt{&\ G''_\hsp{x}\ar@{=}[r]\ar@{>->}[d]&\ G''_\hsp{x}\ar@{>->}[d]\\
\cy{N''}\ \ar@{=}[d]\ar@{>->}[r]&H \ar@{->>}[r]\ar@{->>}[d]& G'\ar@{->>}[d]\\
\cy{N''}\ \ar@{>->}[r]&\cy{N'} \ar@{->>}[r]&\cy{\ell}
}$$
where $G'\subset G$ and $\cy{N'}$ are the images of $H\subset G \times \bZ$ under the first and second factor projection, respectively.
Notice that $\cy{\ell}$ is the common quotient group of $G'$ and $\cy{N'}$. In terms of this data, $H \subseteq G'\times \cy{N'}$ and hence the pre-image $\Gamma_H \subseteq G' \times m\cdot \bZ\subseteq G\times m\cdot \bZ$, where $m=N/N'$.

We now show that $G''$ is a $p$-group. Suppose, if possible, that some other prime $q\neq p$ divides $|G''|$. Since  $H$ is $p$-hyperelementary the Sylow $q$-subgroup of $H$ is cyclic. However  $G''\times \cy{N''} \subseteq H$,  so $q\nmid N''$. But $N' = N''\cdot \ell$, hence this implies
that the $q$-primary part $N'_q = \ell_q \leq |G'|\leq |G|$. Now
$$m = N/N' \geq q^k/N'_q \geq q^k/|G| \geq M$$
by definition of $N = N(M, |G|, p)$. This would imply that $\Gamma_H \subset G \times m\cdot \bZ$ for some $m\geq M$, contrary to our assumption. Hence $P:=G''$ is a $p$-group, or more precisely the $p$-Sylow subgroup of $G'$ since $p\nmid \ell$.

Alternative (ii) is an immediate consequence. If $q\neq p$ is a prime dividing $|H|$, then $q\mid N'$ since $G''$ is a $p$-group. Hence $H$ admits an epimorphism onto a non-trivial finite cyclic $q$-group. By Lemma \ref{lem_hyper_elem_versus_elem}, this implies that $H$ is $p$-elementary.
Note that there is an isomorphism $$j'=(id_P\times s)\colon  P \times \cy \ell \xrightarrow{\cong} G'$$ defined by the inclusion $id_P\colon P\subset G'$ and a splitting $s\colon \cy \ell \to G'$ of the projection $G' \to \cy \ell$.

Next we consider assertion (a). A similar pull-back diagram exists for the subgroup $\Gamma_H \subset G \times \bZ$. We obtain a pull-back diagram of exact
sequences
$$\xymatrix{ 1\ar[r]&P \ar[r]\ar@{=}[d]&\Gamma_H\ar[r]\ar[d]& \bZ\ar[r]\ar[d]& 1\\
1\ar[r]&P \ar[r]&G'\ar[r]& \cy \ell\ar[r]\ar@/_/[l]_s& 1}$$
since $P = \Gamma_H \cap (G \times 0)$, and $\pr_{\bZ}(\Gamma_H) = k\cdot \bZ$ for some positive integer $k$. This exact sequence splits, since it is the pull-back of the lower split sequence: we can choose the element $(g_0, k) \in \Gamma_H \subseteq G'\times \bZ$ which projects to a generator of $\pr_{\bZ}(\Gamma_H) $, by taking $g_0=s(u)$ where $u \in \cy \ell$ is a generator. The isomorphism $\alpha\colon P \times \bZ  \xrightarrow{\approx} \Gamma_H $ is defined by
$\alpha(g, n) = (gg_0^n, n)$ for $g\in P$ and $n \in \bZ$.

Assertion (b) follows by composing the splitting
$ (id_P\times s)\colon P \times \cy \ell \cong G'$  with the inclusion $G'\subseteq G$ to obtain an injection $j\colon  P \times \cy \ell\to G$. By the definition of $g_0$, the composite $(j\times k\cdot id_\bZ)\circ (id_P \times \beta) = i \circ \alpha$, where $i\colon \Gamma_H \to G\times \bZ$ is the inclusion.
\end{proof}

\section{The Proof of Theorem A}\label{four}

We will need some standard results from induction theory for Mackey
  functors over finite groups, due to Dress (see \cite {dress1}, \cite{dress2}), as well as  a refinement called the Burnside quotient Green ring associated to a Mackey functor (see \cite[\S 1]{h2006} for a description of this construction, and \cite{htw2007} for the detailed account).

For any homomorphism $\pr\colon \Gamma \to G$ from an infinite discrete group to a finite group $G$, the functor
$$\cM(H): = K_n(R\Gamma_H),$$
where $\Gamma_H =\pr^{-1}(H)$, is a Mackey functor defined on subgroups $H \subseteq G$. The required restriction maps exist because the index $[\Gamma_H: \Gamma_K]$ is finite for any pair of subgroups $K\subset H$ in $G$.  This point of view is due to
Farrell and Hsiang \cite[\S 2]{farrell-hsiang2}. The Swan ring $SW(G,\bZ)$ acts as a Green ring on $\cM$, and it is a fundamental fact of Dress induction theory that the Swan ring is computable from the family $\cH$ of hyperelementary subgroups of $G$. More precisely, the localized Green ring $SW(G,\bZ)_{(p)}$ is computable from the family $\cH_p$ of $p$-hyperelementary subgroups of $G$, for every prime $p$. If follows from Dress induction that the Mackey functor $\cM(G)_{(p)}$ is also $p$-hyperelementary computable. We need a refinement of this result.

\begin{theorem}[{\cite[Theorem 1.8]{h2006}}] \label{thm: computable}
 Suppose that $\cG$ is a Green ring which
acts on a Mackey functor $\cM$.  If $\cG\otimes \bZ_{(p)}$ is $\cH$-computable, then  every
$x\in \cM(G)\otimes \bZ_{(p)}$ can be written as
$$x = \sum_{H \in \cH_p} a_H \Ind_H^G(\Res_G^H(x))$$
for some coefficients $a_H \in  \bZ_{(p)}$.
\end{theorem}

We fix a prime $p$. For each element $x\in \NK_n(RG)$, let $M= M(x)$ as in Lemma \ref{lem:Frobenius_finally_vanishes} applied to the ring $A=RG$.  Then $$\res_{m} \colon
  K_n(R[G\times \bbZ]) \to K_n(R[G\times \bbZ])$$ sends $i_n(x)$ to
  zero for $m \ge M(x)$.
Now let $N= N(M, |G|, p)$, as defined in Section \ref{three}, and consider $\cM(H) = K_n(R\Gamma_H)$ as a Mackey functor on the subgroups $H \subseteq G \times \cy N$, via the projection $\pr\colon G\times \bZ \to G \times \cy N$.

 Let $\cH_p(x)$ denote the set of $p$-hyperelementary subgroups $H
  \subseteq G \times\cy N$, such that $\Gamma_H$ is \emph{not} contained in
  $G\times m\cdot \bZ$, for any $m\geq M(x)$.
  By the formula of Theorem \ref{thm: computable}, applied to $y = i_n(x)$, we see that $x$ lies in the image of the composite map
  \eqncount
   \begin{eqnarray}
  \bigoplus_{H \in \cH_p(x)} K_n(R\Gamma_H)_{(p)} & \xrightarrow{i_*} &
  K_n(R[G\times \bbZ])_{(p)} \xrightarrow{r_n} \NK_n(RG)_{(p)}.
  \label{x_lies_in_calh_p(M(x)}
  \end{eqnarray}

  We conclude from
  Lemma~\ref{lem:deep_and_p-torsion}~\eqref{lem:deep_and_p-torsion:commutative_square}
  (using that notation) that the composite
   \eqncount
\begin{multline}K_n(R[P \times \bbZ]) \xrightarrow{\alpha_*} K_n(R\Gamma_H)
    \xrightarrow{i_*}  K_n(R[G \times \bbZ]) \xrightarrow{r_n} \NK_n(RG)
  \label{comp(1)}
  \end{multline}
  agrees with the composite
   \eqncount
\begin{multline}
    K_n(R[P \times \bbZ]) \xrightarrow{(\id_P \times \beta)_*} K_n(R[P
    \times \cy \ell\times \bbZ]) \xrightarrow{(j \times \id_{\bbZ})_*}
    K_n(R[G\times \bbZ])
    \\
   \xrightarrow{(\id_G \times k \cdot\id_{\bbZ})_*} K_n(R[G \times
    \bbZ]) \xrightarrow{r_n} \NK_n(RG).
  \label{comp(2)}
  \end{multline}

  Recall that $\beta\colon \bbZ \to \cy \ell\times \bbZ$ sends $n$ to $(nu,n)$ for some generator $u \in \cy \ell$. Let
  $ \NIL(RP) \to \NIL(R[P \times \cy{\ell}])$  be the functor which sends a nilpotent $RG$-endomorphism $f \colon Q \to Q$ of a finitely generated $RP$-module $Q$
  to the nilpotent $R[G \times \cy{\ell}]$-endomorphism
    $$
  R[P \times \cy l] \otimes_{RP} Q \mapsto R[P \times \cy{\ell}] \otimes_{RP} Q,
  \quad x \otimes q \mapsto xu \otimes f(q).$$
 Let $\phi\colon
   \NK_n(RP) \to \NK_n(R[P \times \cy{\ell}])$ denote the induced homomorphism.

  \begin{lemma} \label{lem:three_commutative_diagrams}
  \mbox{}
  \renewcommand{\theenumi}{\alph{enumi}}
 \renewcommand{\labelenumi}{(\theenumi)}
\begin{enumerate}
 \item \label{lem:three_commutative_diagrams:(1)}
  The following diagram commutes
$$\xymatrix@!C=10em{K_n(R[P \times \bbZ]) \ar[r]^-{(\id_P \times \beta)_*}
\ar[d]^-{r_n} & K_n(R[P \times \cy{\ell} \times \bbZ]) \ar[d]^-{r_n}
\\
\NK_n(RP) \ar[r]^-{\phi} & \NK_n(R[P \times \cy{\ell}])\ .}$$

 \item \label{lem:three_commutative_diagrams:(2)}
  The following diagram commutes
  $$\xymatrix@!C=10em{K_n(R[P \times \cy{\ell}\times \bbZ])
  \ar[r]^-{(j\times \id_{\bbZ})_*} \ar[d]^-{r_n} &
  K_n(R[G\times \bbZ]) \ar[d]^-{r_n}
  \\
  \NK_n(R[P \times \cy{\ell}) \ar[r]^-{j_*}& \NK_n(RG)\ .}$$

\item \label{lem:three_commutative_diagrams:(3)}
  The following diagram commutes
  $$\xymatrix@!C=8em{{K_n(R[G\times \bbZ])}
  \ar[r]^{\ind_k}\ar[d]^{r_n}&{K_n(R[G \times \bbZ])}\ar[d]^ {r_n}\\
  {\NK_n(RG)}\ar[r]^{V_k}&{\NK_n(RG)}\ .}$$
  \end{enumerate}
  \end{lemma}
  \begin{proof}
  \eqref{lem:three_commutative_diagrams:(1)}
  The tensor product $\otimes_\bZ$ induces a pairing
  \eqncount
\begin{eqnarray}
\label{naturality_of_pairing}
&\mu_{R, \Gamma}\colon K_{n-1}(R) \otimes_{\bbZ} K_1(\bbZ \Gamma) \to K_n(R\Gamma)&
\end{eqnarray}
for every group $\Gamma$, which is natural in $R$ and $\Gamma$.
  It suffices to prove that the following diagram is commutative for every ring $R$
  (since we can replace $R$ by $RP$). Let $A = R[\cy \ell]$ for short.
  $$\xymatrix@R30mm
  {K_n(R) \oplus K_{n-1}(R) \oplus \NK_n(R) \oplus \NK_n(R)
  \ar[d]_-{\left(\begin{array}{cccc}
  i_{1,1} & 0 & 0 & 0
  \\
  i_{1,2} & i_{2,2} & 0 & 0
  \\
  0 & 0 & \phi & 0
  \\
  0 & 0 & 0 & \phi
  \end{array}\right)}
\ar[r]^-B_-{\cong} & K_n(R\bbZ) \ar[d]^{\beta_*}
\\
K_n(A) \oplus K_{n-1}(A) \oplus \NK_n(A) \oplus \NK_n(A)
\ar[r]^-B_-{\cong} & K_n(A\bbZ)}$$
Here the vertical arrows are the
isomorphisms given by the Bass-Heller-Swan
decomposition~\eqref{Bass-Heller-Swan_decomposition}, the homomorphisms $i_{1,1}$ and $i_{2,2}$ are induced
by the inclusion $R \to R[\cy \ell]$ and the homomorphism $i_{1,2}$ comes from the pairing
 \eqncount
\begin{eqnarray}
&\mu_{R, \cy \ell}\colon K_{n-1}(R) \otimes_{\bbZ} K_1(\bbZ [\cy \ell]) \to K_n(R[\cy \ell])&
\end{eqnarray}
evaluated at the class of the unit $u \in \bbZ[\cy \ell]$ in $K_1(\bbZ[\cy \ell])$ and
the obvious change of rings homomorphisms $K_1(R[\cy \ell]) \to K_1(R[\cy \ell \times \bbZ])$

In order to show commutativity it suffices to prove its commutativity after restricting to
one of the four summands in the left upper corner.

This is obvious for $K_n(RG)$ since induction with respect to group homomorphisms is
functorial.

For $K_{n-1}(R)$ this follows from the naturality of the pairing
(\ref{naturality_of_pairing})
in $R$ and the group $\cy \ell$ and the equality
$$K_1(\beta)(t) = K_1(R[j_{\bbZ}])(t) + K_1(R[j_{\cy \ell}])(u)$$
where $j_{\bbZ} \colon
\bbZ \to \cy \ell \times \bbZ$ and $j_{\cy \ell} \colon \cy \ell \to \cy \ell \times \bbZ$ are the
obvious inclusions.

The commutativity when restricted to the two Bass $\Nil$-groups follows from a result of
Stienstra~\cite[Theorem~4.12 on page~78]{stienstra1}.
\\[1mm]
\eqref{lem:three_commutative_diagrams:(2)} This follows from the naturality in $R$ of
$r_n$.
\\[1mm]
\eqref{lem:three_commutative_diagrams:(3)} It suffices to show that the following diagram
commutes (since we can replace $R$ by $RG$)
$$\xymatrix@R30mm {K_n(R) \oplus K_{n-1}(R) \oplus \NK_n(R) \oplus \NK_n(R)
  \ar[d]_-{\left(\begin{array}{cccc} \id & 0 & 0 & 0
        \\
        0 & k\cdot \id & 0 & 0
        \\
        0 & 0 & V_k & 0
        \\
        0 & 0 & 0 & V_k
  \end{array}\right)}
\ar[r]^-B_-{\cong} & K_n(R\bbZ) \ar[d]^{\ind_k}
\\
K_n(R) \oplus K_{n-1}(R) \oplus \NK_n(R) \oplus \NK_n(R) \ar[r]^-B_-{\cong} &
K_n(R\bbZ)}$$
where the vertical arrows are the isomorphisms given by the
Bass-Heller-Swan decomposition~\eqref{Bass-Heller-Swan_decomposition}.

In order to show commutativity it suffices to prove its commutativity after restricting to
one of the four summands in the left upper corner.

This is obvious for $K_n(R)$ since induction with respect to group homomorphisms is
functorial.

Next we inspect $K_{n-1}(R)$. The following diagram commutes
$$\xymatrix{{K_{n-1}(R)
  \otimes_{\bbZ} K_1(\bbZ[\bbZ])}\ar[r]\ar[d]^{\id \otimes \ind_{k}}&{K_n(R\bbZ)}\ar[d]^ {\ind_{k}}\\
     {K_{n-1}(R) \otimes_{\bbZ}
  K_1(\bbZ[\bbZ])}\ar[r]&{K_n(R\bbZ)}}$$
  where the horizontal pairings are given by $\mu_{R, \bZ}$ from (\ref{naturality_of_pairing}). Since in $K_1(\bZ[\bZ])$, $k$ times the class $[t]$ of the unit $t$ is the class $[t^k]=\ind_k([t])$, the claim follows for $K_{n-1}(R)$.

The commutativity when restricted to the copies of $\NK_n(R)$ follows from
Lemma~\ref{lem:ind/res_and_V/F}. This finishes the proof of
Lemma~\ref{lem:three_commutative_diagrams}.
  \end{proof}

  Lemma~\ref{lem:three_commutative_diagrams} implies that the composite~\eqref{comp(2)}
  and hence the composite~\eqref{comp(1)} agree with the composite
  \begin{multline*}
    K_n(R[P \times \bbZ]) \xrightarrow{r_n} \NK_n(RP) \xrightarrow{\phi} \NK_n(R[P \times \cy \ell])
    \\
    \xrightarrow{\ j_*\ } \NK_n(RG) \xrightarrow{V_k} \NK_n(RG).
  \label{com(3)}
  \end{multline*}
  Since we have already shown that the element $x \in \NK_n(RG)_{(p)}$ lies in the image
  of \eqref{x_lies_in_calh_p(M(x)}, we conclude that $x$ lies in the image of the map
$$
  \Phi = (\Phi_P)\colon
  \bigoplus_{P \in \fP_p(G)}
  \NK_n(RP)_{(p)} \to  \NK_n(RG)_{(p)}
$$
subject only to the restriction $k \in I(g)$ in the definition of $\Phi_P$.

 Consider $k \ge 1$, $P \in \fP_p$ and $g
\in C^{\perp}_GP_p$. We write $k = k_0k_1$ for $k_1 \in I(g)$ and
$(k_0,|g|) = 1$. We have $V_k = V_{k_1} \circ V_{k_0}$ (see
Stienstra~\cite[Theorem~2.12]{stienstra1}).  Since $(k_0,|g|) =
1$, we can find an integer $l_0$ such that $(l_0,|g|) = 1$ and
$(g^{l_0})^{k_0} = g$. We conclude from
Stienstra~\cite[page~67]{stienstra1}
$$V_{k_0} \circ \phi(P,g) = V_{k_0} \circ \phi(P,(g^{l_0})^{k_0})
= \phi(P,g^{l_0}) \circ V_{k_0}.$$
Hence the image of $V_k \circ \phi(P,g)$ is contained in the image of
$V_{k_1} \circ \phi(P,g^{l_0})$ and $g^{l_0} \in C^{\perp}_GP_p$.
This finishes the proof of Theorem A.
\qed

\section{Examples}
We briefly discuss some examples. As usual,  $p$ is a prime and $G$ is a finite group. The first example shows that Theorem A gives some information about $p$-elementary groups.

\begin{example} \label{exa:p-elementary}
Let $P$ be a finite $p$-group and let $\ell \ge 1$ an integer with $(\ell,p) = 1$.
Then Theorem A says that $\NK_n(R[P \times\cy \ell])_{(p)}$ is generated by the images of the maps
$$V_k \circ \phi(P,g) \colon  \NK_n(RP)_{(p)} \to \NK_n(R[P \times\cy \ell])_{(p)}$$
for all $g \in \cy \ell$, and all $k \in I(g)$.
Since the composite $F_k \circ V_k=k \cdot \id$ for $k\geq 1$
(see~\cite[Theorem~2.12]{stienstra1}) and $(k,p) = 1$, the  map
$$V_k \colon \NK_n(RG)_{(p)} \to \NK_n(RG)_{(p)}$$ is injective for all $k \in I(g)$. For $g=1$, $\phi(P, 1)$ is the map induced by the first factor inclusion $P \to P \times\cy \ell$, and this is a split injection.
In addition, the composition of  $\phi(P,g)$ with the functor induced by the projection $P\times \cy \ell \to P$ is the identity on $\NIL(RP)$, for any $g\in \cy \ell$.
 Therefore, all the maps $V_k \circ \phi(P,g)$ are split injective.
It would be interesting to understand better the images of these maps as $k$ and $g$ vary. For example, what is the image of $\Phi_{\{1\}}$ where we take $P = \{1 \}$~?
\qed
\end{example}

In some situations the Verschiebung homomorphisms and
the homomorphisms $\phi(P,g)$ for $g \not=1$ do not occur.

\begin{example} \label{ex:no_phi(P,g)_and_V_k}
Suppose that $R$ is a regular ring.
  We consider the special situation where the $p$-Sylow subgroup $G_p$ is
  a normal subgroup of $G$, and furthermore where $C_G(P) \subseteq P$ holds for
  every non-trivial subgroup $P \subseteq G_p$. For $P \neq \{ 1\}$, we have $C^\perp_G(P) = \{1\}$ and the homomorphism
  $\Phi_P = \phi(P, 1)$, which is the ordinary induction map. We can ignore the map $\Phi_{\{1\}}$ since $\NK_n(R) = 0$ by assumption.
  Therefore, the (surjective) image of $\Phi$ in Theorem A is just the image of the induction map
  $$ \NK_n(RG_p)_{(p)} \to  \NK_n(RG)_{(p)}\ .$$
 Note that $\NK_n(RG_p)$ is $p$-local, and we can divide out the conjugation action on $ \NK_n(RG_p)$ because inner automorphisms act as the identity
 on $\NK_n(RG)$. However, $G/G_p$ is a
finite group of order prime to $p$, so that
$$H_0(G/G_p; \NK_n(RG_p)) = H^0(G/G_p; \NK_n(RG_p)) = \NK_n(RG_p)^{G/G_p}\ .$$
 Hence the induction map on this fixed submodule
 $$\lambda_n\colon \NK_n(RG_p)^{G/G_p} \to  \NK_n(RG)_{(p)}$$
  is surjective.
   An easy application of the double coset formula shows that the composition of $\lambda_n$
with the restriction map $\res_{G}^{G_p} \colon \NK_n(RG)_{(p)} \to \NK_n(RG_p)_{(p)}$ is
given by $|G/G_p|$-times the inclusion $\NK_n(RG_p)^{G/G_p} \to \NK_n(RG_p)$. Since
$(|G/G_p|,p) = 1$ this composition, and hence $\lambda_n$, are both injective.  We conclude that
$\lambda_n$ is an isomorphism.

 Concrete examples are provided by semi-direct products $G = P \rtimes C$, where $P$ is a cyclic $p$-group, $C$ has order prime to $p$, and the action map $\alpha\colon C \to \aut(P)$ is injective. If we assume, in addition, that the order of $C$ is square-free, then
$$\NK_n(\bZ P)^{C} \xrightarrow{\cong} \NK_n(\bZ[P \rtimes C])$$
for all $n\leq 1$ (this dimension restriction, and setting $R=\bZ$, are only needed to apply Bass-Murthy \cite{bass-murthy1} in order to eliminate possible torsion in the $\Nil$-group of orders prime to $p$).
\qed
\end{example}

\section{$NK_n(A)$ as a Mackey functor}\label{seven}
Let $G$ be a finite group.
We want to show that the natural maps
$$i_n \colon \NK_n(RG) \to K_n(R[G\times \bbZ])$$ and
$$r_n \colon K_n(R[G\times \bbZ]) \to \NK_n(RG)$$ in the
Bass-Heller-Swan isomorphism
 are maps of Mackey functors (defined on subgroups of $G$).
Hence  $\NK_n(RG)$ is a direct summand of $K_n(R[G \times \bbZ])$
as a Mackey functor. Since $RG$ is a finitely generated free $RH$-module, for any subgroup $H\subset G$, it is enough to apply the following lemma to $A = RH$ and $B= RG$.

\begin{lemma} \label{lem:Bass_Heller_Swan_and_Mackey}
Let $i \colon A \to B$ be an inclusion of rings. Then the following diagram commutes
$$\xymatrix{K_n(A) \oplus K_{n-1}(A) \oplus \NK_n(A) \oplus \NK_n(A)
\ar[r]^-{\cong} \ar[d]^-{i_* \oplus i_* \oplus i_* \oplus i_*}
&
K_n(A\bbZ) \ar[d]^-{i[\bbZ]_*}
\\
K_n(B) \oplus K_{n-1}(B) \oplus \NK_n(B) \oplus \NK_n(B)
\ar[r]^-{\cong}
&
K_n(B\bbZ)
}$$
 where the vertical maps are given by induction, and
 the horizontal maps are the Bass-Heller-Swan
isomorphisms.
If $B$ is finitely generated and projective, considered as an $A$-module, then
$$\xymatrix{K_n(B) \oplus K_{n-1}(B) \oplus \NK_n(B) \oplus \NK_n(B)
\ar[r]^-{\cong} \ar[d]^-{i^* \oplus i^* \oplus i^* \oplus i^*}
&
K_n(B\bbZ) \ar[d]^-{i[\bbZ]^*}
\\
K_n(A) \oplus K_{n-1}(A) \oplus \NK_n(A) \oplus \NK_n(A)
\ar[r]^-{\cong}
&
K_n(A\bbZ)
}$$
where the vertical maps are given by restriction, and
 the horizontal maps are the Bass-Heller-Swan
isomorphisms.
\end{lemma}

\begin{proof}
One has to show the commutativity of the diagram when restricted
to each of the four summands in the left upper corner.  In each case
these maps are induced by functors, and one shows that the two corresponding composites
of functors are naturally equivalent. Hence the two composites induce the same map on $K$-theory.
As an illustration we do this in two cases.

Consider the third summand $\NK_n(A)$ in the first diagram. The Bass-Heller-Swan
isomorphism restricted to it is given by the the restriction of the map
$(j_+)_* \colon K_n(A[t]) \to K_n(A[t,t^{-1}]) = K_n(A\bbZ)$ induced
by the obvious inclusion $j_+ \colon A[t] \to A[t,t^{-1}]$ restricted to
$\NK_n(A) = \ker \left(\epsilon_* \colon K_n(A[t]) \to K_n(A)\right)$,
where $\epsilon \colon A[t] \to A$ is given by $t = 0$.
Since all these maps come from induction with
ring homomorphisms, the following two diagrams commute
$$\xymatrix{K_n(A[t]) \ar[r]^-{\epsilon_*} \ar[d]^-{i[t]_*} & K_n(A) \ar[d]^-{i_*}
\\
K_n(B[t]) \ar[r]^-{\epsilon_*}  & K_n(B)
}
$$
and
$$\xymatrix{K_n(A[t]) \ar[r]^-{(j_+)_*} \ar[d]^-{i[t]_*} &
K_n(A[t,t^{-1}]) \ar[d]^-{i[t,t^{-1}]_*}
\\
K_n(B[t]) \ar[r]^-{(j_+)_*}  &
K_n(B[t,t^{-1}])
}
$$
and the claim follows.

Consider the second summand $K_{n-1}(B)$ in the second diagram.
The restriction of the Bass-Heller-Swan isomorphism to $K_{n-1}(B)$
is given by evaluating the pairing~\eqref{naturality_of_pairing} for $\Gamma = \bbZ$
 for the unit $t \in \bbZ[\bbZ]$. Hence it suffices
to show that the following diagram commutes, where the horizontal maps
are given by the pairing~\eqref{naturality_of_pairing} for $\Gamma = \bbZ$ and the vertical
maps come from restriction
$$
\xymatrix{K_{n-1}(B) \otimes K_1(\bbZ[\bbZ]) \ar[r] \ar[d]_{i^* \otimes \id}
& K_n(B\bbZ) \ar[d]_{i[\bbZ]^*}
\\
K_{n-1}(A) \otimes K_1(\bbZ[\bbZ]) \ar[r]
& K_n(A\bbZ)
}$$
This follows from the fact that for a finitely generated projective $A$-module
$P$ and a finitely generated projective $\bbZ[\bbZ]$-module $Q$ there
is a natural isomorphism of $B\bbZ$-modules
$$(\res_A P)\otimes_{\bbZ} Q \xrightarrow{\cong}
\res_{A\bbZ} (P \otimes_{\bbZ} Q),
\quad p \otimes q \mapsto p \otimes q. \qedhere$$
\end{proof}
\begin{corollary}
Let $G$ be a finite group, and $R$ be a ring. Then, for any subgroup $H\subset G$,  the induction
maps $\ind_H^G\colon \NK_n(RH) \to \NK_n(RG)$ and the restriction maps $\res^H_G\colon \NK_n(RG) \to  \NK_n(RH)$ commute with the Verschiebung and Frobenius homomorphisms $V_k$, $F_k$, for $k \geq 1$.
\end{corollary}
\begin{proof} We combine the results of Lemma \ref{lem:Bass_Heller_Swan_and_Mackey} with Stienstra's Lemma \ref{lem:ind/res_and_V/F} (note that these two diagrams also commute with $i_n$ replaced by $r_n$).
\end{proof}
\providecommand{\bysame}{\leavevmode\hbox to3em{\hrulefill}\thinspace}
\providecommand{\MR}{\relax\ifhmode\unskip\space\fi MR }
\providecommand{\MRhref}[2]{%
  \href{http://www.ams.org/mathscinet-getitem?mr=#1}{#2}
}
\providecommand{\href}[2]{#2}

\bigskip
\noindent Ian Hambleton\\
    Department of Mathematics \& Statistics\\McMaster University \\Hamilton, ON L8S 4K1, Canada.\\
   E-mail:  ian@math.mcmaster.ca\\

\medskip

\noindent Wolfgang L\"uck\\
    Mathematisches Institut\\Universit\"at M\"unster \\D-48149 M\"unster, Germany.\\
   E-mail:  lueck@math.uni-muenster.de\\

\end{document}